\documentclass{amsart}
\usepackage{graphicx}
\usepackage{amssymb,amscd,amsthm,amsxtra}
\usepackage{latexsym}
\usepackage{epsfig}
\usepackage{mathtools}
\usepackage{esint}
\usepackage{color}

\newtheorem{thm}{Theorem}[section]
\newtheorem{cor}[thm]{Corollary}
\newtheorem{lem}[thm]{Lemma}
\newtheorem{prop}[thm]{Proposition}
\theoremstyle{definition}
\newtheorem{defn}[thm]{Definition}
\theoremstyle{remark}
\newtheorem{rem}[thm]{Remark}
\numberwithin{equation}{section}

\newcommand{\R}{\mathbb R}

\newcommand{\eps}{\epsilon}

\newcommand{\p}{\partial}

\newcommand{\comment}[1]{}

\begin{document}

\title[One-phase almost minimizers]{Almost minimizers of the one-phase free boundary problem}
\author{D. De Silva}
\address{Department of Mathematics, Barnard College, Columbia University, New York, NY 10027}
\email{\tt  desilva@math.columbia.edu}
\author{O. Savin}
\address{Department of Mathematics, Columbia University, New York, NY 10027}\email{\tt  savin@math.columbia.edu}
\begin{abstract}We consider almost minimizers to the one-phase energy functional and we prove their optimal Lipschitz regularity and partial regularity of their free boundary. These results were recently obtained by David and Toro \cite{DaT}, and David, Engelstein, and Toro \cite{DaET}. Our proofs provide a different method based on a non-infinitesimal notion of viscosity solutions that we introduced in \cite{DS}. \end{abstract}

\maketitle
\section{Introduction}
This note is concerned with almost-minimizers of the classical one-phase (Bernoulli) energy functional,
\begin{equation}\label{J}
J(u, \Omega):= \int_{\Omega} (|\nabla u|^2 + \chi_{\{u>0\}}) \ dx,
\end{equation}
with $\Omega$ a bounded domain in $\R^n$ and $u \geq 0.$

Minimizers of $J$ were first investigated systematically by Alt and Caffarelli. Two fundamental questions are answered in the pioneer article \cite{AC}, that is the Lipschitz regularity of minimizers and the regularity of ``flat" free boundaries, which in turns gives the almost-everywhere regularity of minimizing free boundaries. The viscosity approach to the associated free boundary problem was later developed by Caffarelli in \cite{C1,C2,C3}. In particular in \cite{C2} the regularity of Lipschitz free boundaries is obtained. There is a wide literature on this problem and the corresponding two-phase problem, and we refer the reader to the paper \cite{DFS} for a comprehensive survey. 

Almost minimizers of $J$ were investigated recently in \cite{DaT, DaET}.
In \cite{DaT} the authors obtained local Lipschitz continuity  of almost minimizers in the more general case of a two-phase energy functional. Later, in \cite{DaET} the authors proved uniform rectifiability of the free boundary, and in the purely one-phase case they showed that the free boundary is $C^{1,\alpha}$ almost-everywhere. Thus the pioneer results in \cite{AC} have been extended to the context of almost minimizers.

Our purpose here is to provide a different approach, based on non-variational techniques, to study  almost minimizers of $J$ and their free boundaries. 
Our strategy is inspired by our recent work \cite{DS} in which we develop a Harnack type inequality for functions that do not necessarily satisfy an infinitesimal equation but rather exhibit a two-scale behavior. As an application, we provide in \cite{DS} the $C^{1,\alpha}$ estimates of Almgren and Tamanini \cite{A, T} for quasi-minimizers of the perimeter functional, and in \cite{DS1} of the {\it thin} one-phase functional. We follow here the same approach, by showing that almost minimizers of $J$ are ``viscosity solutions" in this more general sense. Roughly, our viscosity solutions satisfy comparison in a neighborhood of a touching point whose size depends on the properties of the test functions (see subsection 4.1). Once this is established, we employ the techniques developed by the first author in \cite{D} to study the regularity of the free boundary of viscosity solutions.

Our first main theorem about the optimal Lipschitz regularity of almost minimizers reads as follows. We refer to Section 2 for the precise definition of almost minimizers (with constant $\kappa$ and exponent $\beta.$)

\begin{thm}\label{LIP} Let $u$ be an almost minimizer for $J$ in $B_1$ (with constant $\kappa$ and exponent $\beta$.) Then
 $$\|\nabla u\|_{L^\infty(B_{1/2})} \le C (\|u\|_{H^1(B_1)} +1)$$ for some constant $C$ depending on $\kappa$, $\beta$ and $n$.
 Moreover, $u$ is uniformly Lipschitz continuous in a neighborhood of $\{u=0\}$, that is if $u(0)=0$ then
$$|\nabla u| \le C(n), \quad \mbox{in $B_{r_0}$},$$
 for some $r_0$ depending on $\kappa, \beta, n$ and $\|u\|_{H^1}$.
\end{thm}

Our next theorem extends the result in \cite{C1} concerning the regularity of the free boundary
$$F(u):=\p \{u>0\} \cap B_1, $$ to the context of almost minimizers. Precisely, we prove an improvement of flatness theorem (see Theorem \ref{flat}), from which the following main regularity result follows.

\begin{thm}\label{thm_intro}
Let $u$ be an almost minimizer to $J$ in $B_1$ (with constant $\kappa$ and exponent $\beta$.) Then
$$\mathcal H^{n-1}(F(u) \cap B_{1/2}) \le C(\beta, \kappa, n),$$
and $F(u)$ is $C^{1,\alpha}$ regular outside a closed singular set of Hausdorff dimension $n-5$, for some $\alpha(\beta,n)>0$ small. 
\end{thm}

Our strategy also allows us to obtain $C^{1,\alpha}$ regularity of Lipschitz free boundaries via the arguments of \cite{D} (see Theorem \ref{thm2}).

The paper is organized as follows. In Section 2 we prove the optimal Lipschitz regularity for almost minimizers, then in section 3 we provide non-degeneracy properties and a compactness result. Section 4 is devoted to the partial regularity of the free boundary.

\section{Lipschitz Continuity of almost minimizers}

In this section we prove Lipschitz continuity of almost minimizers. First, we recall the definition of almost minimizers (see \cite{G} for a comprehensive treatment of almost minimizers of regular functionals of the calculus of variations.)

\begin{defn}\label{almost_E} We say that $u$ is an almost minimizer for $J$ in $\Omega$ (with constant $\kappa$ and exponent $\beta$) if $u \in H^1(\Omega)$, $u \geq 0$ a.e. in $\Omega$, and
\begin{equation}\label{almost_min}
J(u, B_r(x)) \leq (1+ \kappa r^\beta) J(v, B_r(x))
\end{equation}
for every ball $B_r(x)$ such that $\overline{B_r(x)} \subset \Omega$ and every $v \in H^1(\Omega)$ such that $v=u$ on $\p B_r(x)$ in the trace sense.
\end{defn}

Below constants depending only on $n$ are called universal. When $u$ is assumed to be an almost minimizer, then universal constants may depend on $\beta$ as well. Throughout the paper the function $u$ will be non-negative. 

Our first result is the following dichotomy. 

\begin{prop}\label{dich} Let $u \in H^1(B_1)$ and assume that
\begin{equation}\label{first}J(u,B_1) \leq (1+\sigma) J(v,B_1)\end{equation}
for all $v \in H^1(B_1)$ such that $v=u$ on $\p B_1$.
Denote by
\begin{equation}\label{a}a : = \left(\fint_{B_1} |\nabla u|^2 dx\right)^{1/2}.\end{equation}
Given $\eps>0$ small, there exist constants $\eta, M, \sigma_0>0$ (depending on $\eps$) such that if $\sigma \leq \sigma_0$ and $a \geq M$ then the following dichotomy holds. Either
\begin{equation}\label{alt1}\left(\fint_{B_{\eta}} |\nabla u|^2 dx\right)^{1/2} \leq \frac a 2, \end{equation}
or
\begin{equation}\label{alt2}\left( \fint_{B_\eta} |\nabla u - q|^2 dx\right)^{1/2} \leq \eps a,\end{equation}
with $q \in \R^n$ such that 
\begin{equation}\label{slope1}
\frac a 4 < |q| \leq C_0 a,
\end{equation}
and $C_0>0$ universal.
\end{prop}
\begin{proof} Let $v$ denote the harmonic replacement of $u$ in $B_1.$ Then,
$$ \int_{B_1} |\nabla u -\nabla v|^2 \leq J(u, B_1) + \int_{B_1}( |\nabla v|^2 - 2 \nabla u \cdot \nabla v)$$
and by \eqref{first} together with the fact that 
$$\int_{B_1} \nabla v \cdot \nabla(u-v)=0$$ this gives
$$ \int_{B_1} |\nabla u -\nabla v|^2 \leq \sigma \int_{B_1} |\nabla v|^2 + C.$$ Thus, since $v$ minimizes the Dirichlet integral in $B_1$,
$$ \fint_{B_1} |\nabla u -\nabla v|^2 \leq \sigma \fint_{B_1} |\nabla u|^2 + C=\sigma a^2 +C,$$
with $C>0$ universal. 

Since $|\nabla v|^2$ is subharmonic in $B_1$ and $v$ minimizes the Dirichlet integral, we conclude that 
$$|\nabla v| \leq C_0 a \quad \text{in $B_{1/2}$},$$
with $C_0$ universal.
Thus, since $\nabla v$ is harmonic, if we denote by $q:=\nabla v(0)$, we conclude that $|q| \leq C_0a,$
and
$$\fint_{B_{\eta}} |\nabla v - q|^2  \leq C_1 a ^2\eta^2, \quad \quad \forall \, \, \eta \le 1/2,$$
with $C_1$ universal.
Thus,
\begin{equation}\label{b0}\fint_{B_{\eta}} |\nabla u - q|^2 \leq 2\sigma\eta^{-n}a^2 + 2 C_1 \eta^2 a^2 + C \eta^{-n},\end{equation}
and hence 
\begin{equation}\label{bound1}\fint_{B_{\eta}} |\nabla u |^2 \leq 4\sigma\eta^{-n}a^2 + 4 C_1 \eta^2 a^2 + 2C \eta^{-n} + 2|q|^2.\end{equation}
Now, given $\eps>0$, we can choose $\eta$ small (depending on $\eps$) and then $\sigma$ small and $a$ large depending on $\eta$, such that \begin{equation}\label{bound2}4\sigma\eta^{-n}a^2 + 4 C_1 \eta^2 a^2 + 2C \eta^{-n} \leq 2 \eps^2 a^2 \leq \frac{a^2}{8}.\end{equation}

We distinguish two cases.
If $$|q| \leq \frac a 4$$ then \eqref{bound1}-\eqref{bound2} give that
$$\left(\fint_{B_ {\eta}} |\nabla u|^2 dx\right)^{1/2}\leq \frac a 2.$$
Otherwise,
$$\frac a 4 < |q| \leq C_0 a,$$
and \eqref{b0}-\eqref{bound2} give that
$$\left(\fint_{B_{\eta}} |\nabla u - q|^2dx\right)^{1/2} \leq \eps a.$$ This concludes the proof.

\end{proof}

The next Lemma shows that alternative \eqref{alt2} can be ``improved" when $\eps$ and $\sigma$ are small enough.

\begin{lem}\label{iterate} Let $u$ be as in Proposition $\ref{dich}$, with $a \geq a_0>0$. 
Assume that 
\begin{equation}\label{small}\left(\fint_{B_1} |\nabla u -q|^2 dx\right)^{1/2} \leq \eps a\end{equation} for some $\eps >0$, and $q\in \R^n$ such that
\begin{equation}\label{slope}\frac{a}{8} < |q| \leq 2C_0 a,\end{equation} for $C_0>0$ the universal constant in Proposition $\ref{dich}$.

Given $0<\alpha<1$, there exist $\rho=\rho(\alpha)>0$, $\eps_0= \eps_0(\alpha, a_0)$, $c_0=c_0(\alpha, a_0),$ such that if $$\eps \leq \eps_0 \quad \mbox{and} \quad \sigma \leq c_0 \eps^2,$$ then
\begin{equation}\label{step1}
\left(\fint_{B_\rho} |\nabla u - \tilde{q}|^2 dx\right)^{1/2} \leq \eps \rho^{\alpha} a Poincar\ ' e-Sobolev
\end{equation}
with $\tilde q \in \R^n$ such that
\begin{equation}\label{same}|q - \tilde q| \leq \tilde C \eps a,\end{equation} for some $\tilde C>0$ universal.
\end{lem}
\begin{proof} Let $\bar v$ be the harmonic replacement of $u$ in $B_{1/2}$ and denote by $v$ the competitor:
$$v = \bar v \quad \text{in $B_{1/2}$} \quad v=u \quad \text{outside $B_{1/2}$}.$$
Then,
$$J(u, B_1)\leq (1+\sigma)J(v,B_1),$$
that is
$$J(u, B_{1/2}) \leq \sigma J(u, B_1 \setminus B_{1/2}) + (1+\sigma)J(v, B_{1/2}).$$
Notice that from our assumptions ($C$ universal possibly changing from line to line)
$$\fint_{B_1} |\nabla u|^2 dx \leq C a^2$$
thus we conclude that
$$\int_{B_{1/2}} (|\nabla u|^2 - |\nabla v|^2) + |\{u>0\} \cap B_{1/2}| \leq C \sigma(a^2+1) + |B_{1/2}|.$$
Since $v$ is the harmonic replacement of $u$ in $B_{1/2}$, we finally obtain
\begin{equation}\label{uvq}
\int_{B_{1/2}} |\nabla u -\nabla v|^2 dx \leq  C\sigma(a^2+1)+ |\{u=0\} \cap B_{1/2}|.
\end{equation}
We now claim that ($C_1, \delta$ universal,)
\begin{equation}\label{claim}
|\{u=0\} \cap B_{1/2}| \leq C_1 \eps^{2+\delta}.
\end{equation}
Since $v- q \cdot x$ is the harmonic replacement of $u- q \cdot x$ in $B_{1/2}$, we find that $|\nabla v - q|^2$ is subharmonic in $B_{1/2}$ and
$$|\nabla v - q| \leq \tilde C \eps a \quad \text{in $B_{1/4}$}.$$
Thus, if $\bar q$ denotes the gradient of $v- q \cdot x$ at $0$, we conclude that ($C_2$ universal)
\begin{equation}\label{l+l}\fint_{B_{\rho}} |\nabla v - (q+\bar q)|^2 \leq C_2\eps^2 a^2\rho^2, \quad \rho \leq 1/4.\end{equation}
Denote $\tilde q := q +\bar q.$ Then, 
$$|q -  \tilde q| = |\bar q| \leq \tilde C \eps a.$$
Combining \eqref{uvq}-\eqref{claim}-\eqref{l+l} we get,
\begin{equation}\label{combine}
\fint_{B_{\rho}} |\nabla u- \tilde q|^2 dx \leq \bar C \sigma(a^2+1)\rho^{-n}+ 2C_1 \eps^{2+\delta}\rho^{-n}+2C_2\eps^2 a^2 \rho^2.
\end{equation}
We conclude the proof by choosing first $\rho$ depending on $\alpha$ so that $$2C_2\eps^2 a^2 \rho^2 \leq \frac 1 4 \eps^2 a^2 \rho^{2\alpha},$$ then $\eps$ small (depending on $\rho$ and $a_0$) so that
$$2C_1 \eps^{2+\delta}\rho^{-n} \leq \frac 1 4 \eps^2 a_0^2 \rho^{2\alpha}$$
and finally $\sigma$ small depending on $\rho, a_0$ and $\eps$ so that
$$ \bar C \sigma(a^2+1)\rho^{-n} \leq \frac 1 4\eps^2 \rho^{2\alpha}(a^2+a_0^2).$$
We are left with the proof of the claim \eqref{claim}. 
We have 
$$\fint_{B_1} (u-l)dx=0,$$
where $l$ is the linear function $l(x):=b+ q \cdot x$ with $b:=\fint_{B_1} u $. 
Therefore, by Poincar\' e inequality we get that ($c$ universal)
$$c \fint_{B_1} (u-l)^2 dx \leq \fint_{B_1} |\nabla u -\nabla l|^2 dx \leq \eps^2 a^2.$$
Since $u \geq 0$, we conclude that ($C$ universal)
$$ \fint_{B_1}(l^-)^2 dx \leq C \eps^2 a^2.$$
This together with \eqref{slope} gives that for $c_1$ universal ($\eps$ small), 
\begin{equation}\label{large} l \geq c_1 a \quad \text{in $B_{1/2}$}.\end{equation}

After dividing $u$ and $l$ by $a$, we can assume that $a=1$ in \eqref{small} and \eqref{large}. By the Poincar\'e-Sobolev inequality and assumption \eqref{small} we get,
$$\left(\int_{B_1}(u-l)^{2*}dx\right)^{2/2^*} \leq C\eps^2$$
with $l \geq c_1$ in $B_{1/2}$. Our claim immediately follows, with $\delta=4/(n-2).$
\end{proof}

\begin{rem}\label{k=2} We remark that the conclusion of Lemma \ref{iterate} still holds if the lower bound in assumption \eqref{slope} is replaced by the hypothesis
\begin{equation}\label{average}
\fint_{B_1} u \ dx \geq C_1 a, 
\end{equation}
with $C_1$ large enough universal (depending on $C_0$.) This can be easily seen from the proof, as the lower bound is only used when showing that \eqref{large} holds.
\end{rem}

\begin{cor}\label{c2} Let $u$ be an almost minimizer for $J$ in $B_1$ (with constant $\kappa$ and exponent $\beta$) and assume that $u$ satisfies \eqref{small}-\eqref{slope1} and that $a \geq a_0>0$. There exist $\eps_0$, $\kappa_0$ depending on $\beta,n,$ and $a_0$, such that if  $\eps \leq \eps_0$, $\kappa \le \kappa_0 \eps^2$ then \begin{equation}\label{c1}
\|u-l\|_{C^{1,\beta/2}(B_{1/2})} \leq C \eps a,
\end{equation}
with $C$ universal, for some linear function $l$ of slope $q$. Moreover, 
\begin{equation}
\|\nabla u\|_{L^\infty(B_{1/2})} \leq \bar C a, 
\end{equation} with $\bar C$ universal.
\end{cor}

\begin{rem}\label{r2}
From \eqref{c1} we obtain that $\nabla u \ne 0$, hence $u>0$ in $B_{1/2}$.
\end{rem}
\begin{proof} We show that we can iterate Proposition \ref{iterate} indefinitely with $\alpha=\beta/2$. Indeed, if ($q_0:=q$)
\begin{equation}\label{200}
\left(\fint_{B_{r}} |\nabla u - q_k|^2 dx\right)^{1/2} \leq \eps r^{\beta/2}a, \quad \mbox{with} \quad r=\rho^k,
\end{equation}
then the rescaling $u_r(x):=r^{-1}u(rx)$ satisfies the hypotheses of Proposition \ref{iterate} with 
$$\sigma_r= \kappa r^ \beta, \quad \eps_r:= \eps r^{\beta/2}.$$
Moreover, $$|q_{i+1}-q_{i}| \le C \eps \rho^{i \beta/2} a, \quad i \le k-1, \quad \mbox{and} \quad \frac a4 \le |q_0| \le C_0 a,$$ guarantee that \eqref{slope} is always satisfied if $\eps_0$ is sufficiently small. By Proposition \ref{iterate}, \eqref{200} is satisfied also for $r=\rho^{k+1}$, and therefore it is true for all $k$'s.

After relabeling $\eps$ by $C \eps$ if necessary, the same conclusion holds for all balls $B_r(x) \subset B_{3/4}$.

By standard (Campanato) estimates, we deduce that
$$\|\nabla u- q_0\|_{C^{0,\beta/2} (B_{1/2})} \leq C \eps a,$$ from which the desired claims easily follow. 
\end{proof}

We are now ready to prove the main result Theorem \ref{LIP}.
We remark that in the statement of the theorem we can replace $\|u\|_{H^1}$ by $\|u\|_{L^2}$ as it is known that almost minimizers still satisfy a Caccioppoli type inequality (see Theorem 6.5, Giusti \cite{G}).

\

\noindent {\it Proof of Theorem $\ref{LIP}.$} After an initial dilation of factor $s^{-1}$, $s$ small, we can assume that $u$ is an almost minimizer with a constant $\tilde \kappa= \kappa s^{\beta}$, that can be made arbitrarily small.
 
For $\alpha=\beta/2$, and $a_0 =1$, let $\eps_0= \eps_0(\beta, 1), c_0=c_0(\beta, 1)$, be given by Lemma \ref{iterate}. Now, let $\eta$, $M \ge 1$, and $\sigma_0$ be the constants from Proposition \ref{iterate} associated to $\eps=\eps_0$. Then, set
$$a(\tau):= \left(\fint_{B_\tau} |\nabla u|^2 dx\right)^{1/2}.$$
We consider the integers $k \ge 0$ for which the following inequality holds
\begin{equation}\label{ki}
a(\eta^k) \leq C(\eta) M + 2^{-k} a(1),
\end{equation}
with $C(\eta)$ a large constant.

For $k=0$ this is clearly satisfied. 
If it holds for all $k$'s then it easily follows that $$a(r) \le C(M, \eta) (1+ a(1)), \quad \forall r <1.$$
Otherwise let $k+1$ be the first integer for which \eqref{ki} fails.  
If $$a(\eta^{k}) \leq M,$$ then \eqref{ki} holds also for $k+1$ since $$a(\eta^k) \leq C(\eta) a(\eta^{k-1}),$$ and we reach a contradiction. Thus $$a(\eta^{k}) >M,$$ and, according to Proposition \ref{dich} (rescaled) we get that either
$$a(\eta^{k+1}) \leq \frac 1 2 a(\eta^{k}),$$
which gives again a contradiction, or
\begin{equation}\label{201}\left( \fint_{B_{\eta^{k+1}}} |\nabla u - q|^2 dx\right)^{1/2} \leq \eps_0 a(\eta^{k}),\end{equation}
with 
$$\frac 14  a(\eta^{k}) < |q| \leq C_0 a(\eta^{k}).$$
According to Corollary \ref{c2} we find $$a(r) \le \bar C a(\eta^{k}) \le C(M,\eta) (1 + a(1)) \quad \mbox{ for all} \quad  r \le \eta^{k}.$$

In conclusion $a(r) \le C (1+ a(1))$. The same inequality can be obtained for the averages over all balls with center in $B_{1/2}$ which are included in $B_1$ which gives $$\|\nabla u\|_{L^{\infty}(B_{1/2})} \le C(1+a(1)),$$ 
by Lebesgue Differentiation Theorem. 

If $u(0)=0$, then by Remark \ref{r2} we see that we can never end up in the alternative \eqref{201}. This means that in this case \eqref{ki} holds for all $k \ge 0$, hence $$a(\eta^k) \le C,$$ is uniformly bounded for a sufficiently large $k$. By the result above this implies that $u$ is uniformly Lipschitz continuous in $B_{\eta^k/2}$. 
\qed

\section{Non-degeneracy} The purpose of this section is twofold. First we show that almost minimizers are well approximated by harmonic functions (in their positivity set). Then, we use this fact to obtain non-degeneracy properties of almost minimizers, 
which are a crucial ingredient to use compactness arguments.

We assume throughout this section that 
 $u$ satisfies
\begin{equation}\label{3000}
\|\nabla u\|_{L^\infty(B_1)} \le K, \quad \quad \mbox{and} \quad J(u,B_1) \leq J(v, B_1) + \sigma
\end{equation} 
for any $v \in H^1(B_1)$ which agrees with $u$ on $\p B_1$. In what follows constants $c,C$ may depend on $K$, and in the body of the proof they possibly change from line to line. Dependence on $K$ is often omitted. A constant depending only on $n$ is called universal.

First we remark that the second inequality in \eqref{3000} follows from the condition of almost minimality. Precisely, 
since $u$ is Lipschitz  $J(u,B_1) \le C$. Then the energy inequality
\begin{equation}\label{en1}
J(u,B_1) \le (1+\sigma) J(v,B_1)
\end{equation} 
for any $v$ that equals with $u$ on $\p B_1$, implies
\begin{equation}\label{en2}
J(u,B_1) \le  J(v,B_1) + C' \sigma,
\end{equation}
for some $C'$ large enough. It is more convenient working with \eqref{en2}, hence with \eqref{3000} after relabeling $\sigma$, instead of \eqref{en1} since the energies cancel in a region where $v=u$ and the inequality behaves better with respect to scaling (see remark below.)

\begin{rem}\label{rem01}We remark that the rescaling $ u_\rho(x):=u(\rho x)/\rho$ satisfies \eqref{3000} with $\sigma_\rho:= \rho^{-n} \sigma$.
\end{rem}

\subsection{Approximation by harmonic functions.} We prove first the following basic lemma, which compares $u$ with its harmonic replacement.

\begin{lem}\label{replace}Let $u$ satisfy \eqref{3000} and let $B_1 \subset \{u>0\}$. Denote by $v$ the harmonic replacement of $u$ in $B_1$. Then, 
\begin{equation}\label{uv}
\|u-v\|_{L^\infty(B_{1/2})} \leq c(\sigma), \quad c(\sigma) \to 0 \quad \text{as $\sigma \to 0$.}
\end{equation}
\end{lem}
\begin{proof} By the maximum principle, $v>0$ in $B_1$, hence using \eqref{3000} and the fact that $v$ is the harmonic replacement of $u$ we get 
$$\int_{B_1}|\nabla u -\nabla v|^2 dx \leq \sigma.$$

By Poincar\' e inequality we conclude that ($C$ changing from line to line)
$$\int_{B_{3/4}} (u-v)^2 dx \leq C\sigma$$
with $u-v$ uniformly Lipschitz in $B_{3/4}$. Thus, if $(u-v)(x) \geq \mu$ say at $x \in B_{1/2}$, we conclude that 
$$ \mu^{2+n} \leq C \sigma,$$ and the conclusion holds with $c(\sigma)= C \sigma^{1/(n+2)}.$
\end{proof}

A consequence of Lemma \ref{replace} is the following version of Harnack inequality.

\begin{lem}\label{lhi} Let $u$ satisfy \eqref{3000} and 
assume $B_1 \subset \{u>0\}$. Let $w$ be a harmonic function such that $u \ge w$ in $B_1$ and $u - w \ge \mu$ at $0$ for some $\mu \le \mu_0$, $\mu_0$ small depending on $K$. Then $u-w \ge c \mu$ in $B_{1/2}$ for some $c$ universal provided that $\sigma \le \mu^{n+3}$. 
\end{lem}

A similar statement clearly holds if $w$ lies above $u$ and separates strictly from $u$ at $0.$

\subsection{Non-degeneracy.} In this subsection we state and prove the non-degeneracy lemmas.

\begin{lem}[Weak non-degeneracy] \label{weaknd}Assume that $u$ satisfies \eqref{3000} for $\sigma$ small and $B_1 \subset \{u>0\}$.Then $u(0)\ge c$ for some $c=c(K)>0$. 

\end{lem}

\begin{proof}Let $v$ be the harmonic replacement of $u$ in $B_1$. Then according to Lemma \ref{replace},  it is enough to prove the desired statement for $v.$

Now, let $\varphi \in C_0^{\infty}(B_{1/2})$ with $\varphi \equiv 1$ in $B_{1/2}$ and $0 \leq \varphi \leq 1.$ Since $v$ minimizes the Dirichlet integral and $v>0$ in $B_1$ (by the maximum principle), we have
\begin{equation}\label{Juv}J(v, B_1) \leq J(u, B_1) \leq J(v(1-\varphi), B_1) + \sigma.\end{equation}
On the other hand, since $v$ is harmonic in $B_1$, ($C$ universal possibly changing from line to line)
\begin{equation*}
\|v\|_{L^\infty(B_{1/2})}, \|\nabla v\|_{L^\infty(B_{1/2})} \leq C v(0),
\end{equation*}
from which we deduce that
\begin{equation}\label{cv}
\int_{B_1} |\nabla v|^2 dx \geq \int_{B_1} |\nabla v(1-\varphi)|^2 dx -C (v(0))^2.
\end{equation}
Combining \eqref{Juv}-\eqref{cv} (since $v>0$ in $B_1$) we get
$$\int_{B_1}|\nabla v| ^2 + |B_1| \leq C(v(0))^2 + \int_{B_1} |\nabla v|^2 + |B_1|-|B_{1/4}| +\sigma,$$
and the desired claim follows for $\sigma$ small.
\end{proof}

\begin{lem}[Strong non-degeneracy]\label{snd}Let $u$ satisfy \begin{equation}\label{3001}
\|\nabla u\|_{L^\infty(B_1)} \le K, \quad \quad \mbox{and} \quad J(u,B_r) \le (1+\sigma) J(v, B_r) 
\end{equation} 
for any ball $B_r(x)$ with $ \overline{B_r(x)} \subset B_1$, and any $v \in H^1(B_r(x))$ which agrees with $u$ on $\p B_r(x)$.
If $0 \in \p \{u>0\}$ and $\sigma$ is small enough, then for some $c=c(K)>0$ $$\max_{B_r} u \ge c r.$$ 

\end{lem}
\begin{proof}The proof of this result is standard once the following claim is obtained (see for example \cite{C3}). 

{\it Claim:} Let $x_0 \in B_1 \cap \{u>0\}$ (close to the origin). There exists a sequence $x_k \in B_1$ such that ($C$ depending on $K$)
\begin{equation}\label{delta}
u(x_{k+1})=(1+ \delta) u(x_k)
 \end{equation}
with
\begin{equation}\label{delta2}
|x_{k+1}-x_k| \leq C dist(x, \p \{u>0\})
 \end{equation}
for some $\delta$ small (depending on $K$). 

Now the claim follows from Lemma \ref{lhi}, in view of the discussion following equation \eqref{3000}.
\end{proof}

Once Lipschitz continuity and non-degeneracy have been established, it is straightforward to show that any blow-up sequence must converge uniformly on compact sets to a global minimizer. Moreover, by the Weiss monotonicity formula \cite{W}, it follows that the limit must be homogenous of degree 1. We sum up these facts in the next proposition.

\begin{prop}\label{blowup}
Assume $u$ is an almost minimizer of $J$ with constant $\kappa$ and exponent $\beta$, and that $0 \in \p \{u>0\}$. Any blow up sequence converges uniformly (up to subsequences) to a global minimizing cone (homogeneous of degree 1) which has $0$ as a free boundary point. Also, their free boundaries converge in the Hausdorff distance (on compact sets) to the free boundary of the cone.
\end{prop}

\begin{rem}
Using the fact that global minimizers have Lipschitz constant less than $1$ (see \cite{CJK}), one can show via Lemma \ref{iterate} and Remark \ref{k=2}, that almost minimizers have Lipschitz constant less than 2 in a $r_0$ neighborhood around a free boundary point with $r_0$ depending on $n$, $\kappa$, $\beta$ and $\|u\|_{L^2}$. In other words we can take $K=2$ in statement of Theorem \ref{LIP}.
\end{rem}

\section{Partial regularity of the free boundary} 

\subsection{Almost minimizers as viscosity solutions}

In this subsection we show that almost minimizers satisfy the comparison principle with appropriate families of sub and 
supersolutions of the classical one-phase free boundary problem. The difference with the infinitesimal case is that,  in order to reach a contradiction   
we need to specify the size of the neighborhood around the contact point between the solution and an explicit barrier.

\begin{lem}[Subsolution]\label{sub_property}
Let $u$ satisfy \eqref{3000} and let $P$ be a quadratic polynomial such that 
$$\|D^2 P\| \le 1, \quad \triangle P \ge \mu,$$
for some $\mu \le \mu_0$ small. Assume that 
$$\mbox{either $u>0$ or $|\nabla P| \ge 1+\mu$ in $B_1$.}  $$
Then $P$ cannot be below $u$ in $B_1$ and touch $u$ by below at a point in $B_{1/2}$ if $\sigma \le \mu^{n+3}$.
\end{lem}

We remark that the subsolution Lemma \ref{sub_property} is used to provide a comparison principle for a function $u$ satisfying \eqref{3000}. Precisely, if $P$ as in Lemma \ref{sub_property} is such that $|\nabla P| \geq 1+ \mu$ and $u \geq P$ in say $B_1 \setminus B_{1/2}$, we can conclude that $u \geq P$ in $B_{1/2}.$
More generally, one way to apply Lemma \ref{sub_property} is given in the following form of the comparison principle for almost minimizers.

\begin{cor}[Comparison principle]\label{co_1}
Let $u$ satisfy \eqref{3000} and
 $$\mbox{$u \ge P$ in a $\delta$-neighborhood of $\partial \mathcal U$ of some domain $\mathcal U \subset B_1$,}$$ 
 for some quadratic polynomial $P$ with $\|D^2 P\| \le \delta^{-1}$, $\triangle P \ge \mu$. Assume that in $\mathcal U$ either $u>0$ or $|\nabla P| \ge 1+\mu$. If $$\mu^{n+3} \ge C(K, \delta) \, \, \sigma, \quad \mbox{ then} \quad u \ge P \quad \mbox{in $\mathcal U$}.$$
\end{cor}

\begin{proof}
Otherwise (a vertical translation of) $P$ touches $u$ by below at some point $x_0 \in \mathcal U$ at distance greater then $\delta$ from $\p \mathcal U$. Now we rescale this picture from $B_\delta(x_0)$ to $B_1$ and contradict Lemma \ref{sub_property}. Indeed, the rescaled functions $\tilde u$, $\tilde P$ satisfy the hypotheses of lemma in $B_1$ with $\tilde P$ touching $\tilde u$ by below at the origin and with constants (see Remark \ref{rem01}) $$\tilde \sigma=\delta^{-n} \sigma \quad, \quad  \tilde \mu = \mu \delta.$$
\end{proof}

Similarly we obtain the following viscosity supersolution lemma which gives a version of the Corollary above (comparison principle) for polynomials $P$ lying above $u$.

\begin{lem}[Supersolution]\label{sup_property} Let $u$ satisfy \eqref{3000}
and let $P$ be a quadratic polynomial such that 
$$\|D^2 P\| \le 1, \quad \triangle P \le - \mu,$$
for some $\mu>0$ small. Assume that 
$$\mbox{either $u>0$ or $|\nabla P| \le 1-\mu$ in $B_1$.}  $$
Then $P^+$ cannot be above $u$ in $B_1$ and touch $u$ by above at a point in $B_{1/2} \cap \overline{\{P>0\}}$ if $\sigma \le \mu^{n+3}$.
\end{lem}

Next we provide the proofs of Lemma \ref{sub_property} and \ref{sup_property}.

\

{\it Proof of Lemma $\ref{sub_property}$.}
If $u>0$ then the conclusion follows easily from Lemma \ref{replace}. 

We assume that $|\nabla P| \ge 1+ \mu$, $u \ge P$ in $B_1$ and $u(x_0)=P(x_0)$ for some $x_0$ in $B_{1/2}$. Since $D^2 P$ is bounded, $u$ is Lipschitz and $P$ touches $u$ by below at $x_0$, we find that $P$ is uniformly Lipschitz continuous in $B_1$. Let 
$$\bar P(x):=P(x) + \frac{\mu}{4n}(1-|x|^2),$$
so that $\triangle \bar P \ge \frac \mu2$, $|\nabla \bar P| \ge |\nabla P|- \frac \mu 2$ and by the Lipschitz continuity of $u$ and $\bar P^+$,
\begin{equation}\label{204}
\bar P^+ - u \ge \frac{ \mu}{8n} \quad \mbox{in} \quad B_{c \mu}(x_0).
\end{equation}
Set $$u_{max}:=\max\{u, \bar P^+\}, u_{min}:=\min\{u, \bar P^+\}$$ and notice that 
$$u_{max}= u, \quad u_{min}=\bar P^+ \quad \text{on $\p B_1$}.$$
Then we have
$$J(u, B_1) \le J(u_{max},B_1) + \sigma,$$
or equivalently
\begin{equation}\label{umin} J(u_{min}, B_1)-J(\bar P^+, B_1)  \le  \sigma.\end{equation}
We claim that 
$$J(u_{min}, B_1)-J(\bar P^+, B_1) \geq \frac \mu 2 \int_{B_1}(\bar P^+-u_{min}) \ dx.$$
Combining this claim with \eqref{204} and \eqref{umin} we conclude that
$$c \mu^{n+2} \leq  \sigma,$$ and get a contradiction for $\sigma \leq \mu^{n+3}$ and $\mu$ small.

We are left with the proof of the claim. For this we minimize the functional 
$$\int_{B_1}|\nabla v|^2 + \chi_{\{v>0\}} + \frac \mu 2(v - \bar P^+) \, \, dx$$
among all competitors $0 \leq v \leq \bar P^+$ which coincide with $\bar P^+$ on $\p B_1$, and claim that $\bar P^+$ is the minimizer. This will imply the claim as $u_{min}$ is an admissible competitor.

In the region where the minimizer $v$ is strictly below $\bar P^+$, it satisfies (in the viscosity sense)
$$\Delta v= \frac{\mu}{4} \quad \text{in $\{v>0\}$,}  \quad |\nabla v|^2= 1 \quad \text{on $F(v)$}.$$
 This means that $v$ satisfies the comparison principle with the continuous family of classical subsolutions $(\bar P+ t)^+$ as we increase $t$ from a large negative constant up to $t = 0$, and we obtain $v \equiv \bar P^+$ as desired. Here we used that $\Delta \bar P > \frac \mu 4$ and $|\nabla \bar P| \geq 1+\frac{\mu}{2} > 1.$ 

\qed

{\it Proof of Lemma \ref{sup_property}.}
The proof follows the lines of the previous lemma, thus we only sketch the argument. As before it suffices to assume $|\nabla P| \le 1- \mu$, $u \leq P^+$ in $B_1$ and $u(x_0)=P^+(x_0)$ with 
$x_0 \in B_{1/2} \cap \overline{\{P>0\}}.$
Denote 
$$\bar P:= P(x) - \frac{\mu}{4n}(1-|x|^2).$$
The energy inequality reads 
$$J(u, B_1) \le J(u_{min},B_1) +\sigma,$$
hence
\begin{equation}\label{comp22}J(u_{max},B_1)-J(\bar P^+, B_1) \leq  \sigma,\end{equation}
where
$$u_{min}:=\min\{u, \bar P^+\}, \quad u_{max}:= \max\{u, \bar P^+\}.$$
As above, we can show that $\bar P^+$ is the minimizer of the functional
$$\int_{B_1}\left( |\nabla v|^2 + \chi_{\{v>0\}} + \frac \mu 2(\bar P^+ -v) \right)\, \, dx,$$
among all competitors $v \geq \bar P^+$ which coincide with $\bar P^+$ on $\p B_1$, by using that $(\bar P + t)^+$ with $t \ge 0$ is a continuous family of classical comparison supersolutions.

This implies that 
$$J(u_{max},B_1)-J(\bar P^+, B_1)\geq \frac \mu 2 \int_{B_1} (u_{max} - \bar P^+) \ dx.$$
It remains to show that the right hand side is greater than $c \mu ^{n+2}$.

If $u(x_0)< \frac{\mu}{32n}$, we can easily conclude from the fact that $u$ touches $P^+$ at $x_0$ and $|\nabla P|\le 1$, that $P < \frac{\mu}{16n}$ in $B_{\frac{\mu}{32}n}(x_0).$
This implies that  $\bar P^+ \equiv 0$ in $B_{\frac{\mu}{32n}}(x_0)$, then by the Lipschitz continuity and the nondegeneracy of $u$, we conclude that 
$$u_{max} - \bar P^+= u \geq c(K) \mu, \quad \text{in $B\subset B_{\frac{\mu}{32n}}(x_0)$},$$
with $|B| \sim \mu^n.$ 

If $u(x_0) \geq \frac{\mu}{32n}$, then using that $u(x_0)=P(x_0)$, the Lipschitz continuity of $u$ and $\bar P^+$, we get that 
$$u_{max}-\bar P^+ = u -\bar P^+\geq c \mu \quad \text{in $B_{c'\mu}(x_0)$}.$$
\qed

\subsection{Partial regularity of the free boundary} In this subsection we prove our main regularity result for the free boundary of almost minimizers. In the context of minimizers this result is contained in \cite{AC}. The general case of  viscosity solutions is due to \cite{C1,C2}, with a different proof provided in \cite{D}. Our proof for almost minimizers will rely on the techniques in \cite{D}.

We have the following theorem.

\begin{thm}[Flatness implies regularity.] \label{flat} Let $u$ be an almost minimizer to $J$ in $B_1$ (with constant $\kappa$ and exponent $\beta$), and let $|\nabla u| \leq K$.
Assume $|u-x_n^+| \le \eps_0$ in $B_1$ and $0 \in F(u):=\p\{u>0\} \cap B_1.$ If $\eps_0$ and $\kappa$ are small enough depending on $\beta$ and $K$, then $F(u)$ is $C^{1,\alpha}$ in a neighborhood of $0$, for some $\alpha \leq \beta/(n+4)$.
\end{thm}

Theorem \ref{flat} follows easily from the improvement of flatness lemma below. Its proof is presented at the end of the section.

\begin{lem}\label{limp5}Let $u$ satisfy \eqref{3000}.
Assume that $|u-x_n^+| \le \eps$ in $B_1$, $0 \in F(u)$, and $\sigma$ in \eqref{3000} satisfies $\sigma \le \eps^{n+4}$. Given $\alpha \in (0,1)$ there exists $\eta$ depending on $\alpha$ such that 
\begin{equation}\label{eta}|u-(x \cdot \nu)^+| \le \eps \eta^{1+\alpha} \quad \quad { in} \quad  B_\eta\end{equation} for some unit direction $\nu$, provided that $\eps \le \eps_0(K, \alpha)$ is sufficiently small.
\end{lem}

As mentioned before, the results in subsection 4.1 guarantee that the proof of Lemma \ref{limp5} follows along the lines of  the case of minimizers as in \cite{D}. We sketch the details in the following two subsections.

\subsection{Two properties}Define the $\eps$-scaled function
$$\bar u_\eps:= \frac 1 \eps (u - x_n) \quad \mbox{in the set $\{ u>0\} \cap B_1$.} $$
Next we state two properties (P1) and (P2) for the function $u$ which turn out to be sufficient for obtaining the approximation of $\bar u_\eps$ with solutions of the linearized Neumann problem
\begin{equation}\label{Ne}\triangle \bar u_0 =0 \quad \text{in $B_{1/2}^+$}, \quad \quad \partial _n \bar u_0=0 \quad \mbox{on $\{x_n=0\} \cap B_{1/2}$},\end{equation}
and for obtaining the improvement of flatness Lemma \ref{limp5}. These properties are written in terms of two small parameters $\delta$ and $\eps>0$. 

\

(P1) {\it Harnack inequality,} (see Lemma $3.3$ in \cite{D}.)

Given $\delta>0$, there exists $\eps_0=\eps_0(\delta)$ such that if $\eps \le \eps_0$ and for some constant $a$, $$u \ge l^+=(x_n+a)^+, \quad \mbox{ in $B_r(x_0) \subset B_1$,}$$ with $r \ge \delta$, $|a| \le \eps$ and 
$$\mbox{$u(y) \ge l^+(y) + \gamma \eps$ for some $y$ for which $B_{r/2}(y) \subset \{l^+>0\} \cap B_r(x_0)$},$$ and some $ \gamma \in [\delta,1] $, then $$u \ge (x_n + a + c \gamma \eps)^+ \quad \mbox{ in $B_{r/2}(x_0)$},$$ for some $c>0$ universal.

Similarly, the above holds when we replace $\ge$ by $\le$ and $\gamma$ by $-\gamma$.

\

(P2) {\it Viscosity property.}
Given $\delta>0$, there exists $\eps_0=\eps_0(\delta)$ such that if $\eps \leq \eps_0$
we cannot have $u(x_0)=P(x_0)$ and $u \ge P$ in $B_{\delta}(x_0) \subset B_1$  where $P$ is a quadratic polynomial such that $\|D^2 P\| \le \delta^{-1} \eps$, $\triangle P \ge \delta \eps$, and in the ball $B_{\delta}(x_0)$ either $u>0$ or $|\nabla P|>1+\delta \eps$.

Similarly, the above holds when $u \leq P$, $\Delta P \leq -\delta \eps$ and $|\nabla P| < 1-\delta\eps$.
\medskip

We explain why (P1) and (P2) suffice to obtain the improvement of flatness property as in \cite{D}. 

\begin{lem}\label{limp6}
Assume a family of functions $u$ satisfy properties (P1) and (P2) above. If $|u-x_n^+| \le \eps$ in $B_1$, $0 \in F(u)$, then 
$$|u-(x \cdot \nu)^+| \le \eps \eta^{1+\alpha} \quad \quad { in} \quad  B_\eta,$$
for some unit direction $\nu$, provided that $\eps \le \eps_1$ with $\eps_1$ depending on $n$, $\alpha$ and the dependence $\delta \mapsto \eps_0(\delta)$ that appears in properties $(P1), (P2).$

\end{lem}

\begin{proof}
The lemma follows by compactness.  
We argue by contradiction and produce sequences $\eps_k \to 0$ and a sequence of functions $u_k$ satisfying the assumptions but not the conclusion.
Let
$$\bar u_k:= \frac{1}{\eps_k} (u_k - x_n) \quad \mbox{in the set $\{ u_k>0\} \cap B_1$.} $$
Then if we choose $\eps_k \leq \eps_0(\delta_k)$ with $\delta_k=2^{-k}$, property (P1) together with the Ascoli-Arzela theorem guarantee that (up to a subsequence) the graphs of $\bar u_{k}$ converge in the Hausdorff distance to the graph of a H\"older function $\bar u_0$ defined in the half ball $B_{1/2}^+$.

Next we show that property (P2) implies that the function $\bar u_0$ satisfies 
\begin{equation*}\label{Ne1}\triangle \bar u_0 =0 \quad \text{in $B_{1/2}^+$}, \quad \quad \partial _n \bar u_0=0 \quad \mbox{on $\{x_n=0\} \cap B_{1/2}$},\end{equation*} in the viscosity sense. Indeed, if $Q$ is a quadratic polynomial with $\triangle Q>0$ that touches $\bar u_0$ by below at some point $x_0 \in B_{1/2}^+$ then $Q+c_k$ touches by below $\bar u_k$ by below at $x_k \to x_0$ in a fixed neighborhood of $x_0$. Thus we can find $\delta>0$ small such that $\|D^2 Q \| \le \delta^{-1}, \quad \triangle Q > \delta$ and
$$P_k:=x_n + \eps_k (Q + c_k) $$ touches $u_k$ by below at $x_k$ and it is below it in a $\delta$ neighborhood of $x_k$. Notice that $P_k>0$ in this neighborhood and thus we contradict property (P2) for all large $k$. 

If we touch $\bar u_0$ by below at some boundary point $x_0 \in B_{1/2} \cap \{x_n=0\}$ by a quadratic polynomial $Q$ that satisfies in addition $Q_n (x_0)>\delta$ then we argue as above and find that $P_k$ touches $u_k$ by below in $\overline{\{u_k >0\}} \cap B_{2 \delta} (x_0)$ at $x_k \to x_0$. Since $P_k$ is increasing in the $x_n$ direction we conclude that $P_k$ is below $u_k$ in a whole $\delta$-neighborhood of $x_k$ and we contradict (P2) again since $|\nabla P_k| \ge \partial _n P_k \ge 1 + \delta \eps_k$.  

Now the conclusion of the lemma follows easily from the $C^2$ estimates for solutions to the Neumann problem \eqref{Ne} for $\bar u_0$, since $\|\bar u_0\|_{L^\infty} \le 1$,
$$|\bar u_0 - l| \le \frac 12 \eta^{1+ \alpha} \quad \mbox {in $B^+_ \eta$,}$$
for some linear function $l$ and with $\eta>0$ small depending only on $\alpha$ and $n$.

\end{proof}
\subsection{Properties (P1) and (P2) are satisfied.} In order to establish Lemma \ref{limp5} it suffices to show that if $\sigma \le \eps^{n+4}$ then properties (P1) and (P2) hold.
For (P1), by Remark \ref{rem01} it suffices to consider the case $B_r(x_0)=B_1$ and replace $\sigma$ by $$\bar \sigma = r^{-n} \sigma \le  \delta^{-n}\eps^{n+4}.$$
Assume $a=0$ for simplicity. Lemma \ref{lhi} applies since $\bar \sigma \le (\gamma \eps)^{n+3}$ for $\gamma \ge \delta$ and $\eps$ is small depending on $\delta$. We obtain 
\begin{equation}\label{c0}u \ge x_n^+ + c \gamma \eps \quad \text{in $B_{3/4} \cap \{x_n \ge c_0(n)\}.$}\end{equation} Now we can use Lemma \ref{sub_property} (with $\mu=c \gamma \eps/4$) and show that $u$ must be greater than 
$$P=x_n + \frac c2 \gamma \eps (c_0+x_n+2n x_n^2 -|x'|^2),$$
in the cylinder $\mathcal C:=\{|x_n| \le 2 c_0, \quad |x'| \le1/4  \}$, thus the desired conclusion easily follows.
Indeed in view of Corollary \ref{co_1}, if $u \geq P$ in a neighborhood of $\p \mathcal C$ and $|\nabla P| > 1+ \mu$
then $u \geq P$ in $\mathcal C.$ In the region where $c_0 \leq x_n \leq 2 c_0$ the fact that $u \geq P$ immediately follows from \eqref{c0}, if $c_0$ is small enough depending on $n$. In the remaining region we use that $u \geq x_n$ and choose the neighborhood sufficiently small.

\medskip

Finally property (P2) follows from Lemma \ref{sub_property}, since after a rescaling of factor $\delta^{-1}$, the rescaled polynomial $\tilde P$ satisfies
$$\|D^2\tilde P\| \le \eps, \quad \triangle \tilde P \ge \delta^2 \eps, \quad \tilde P>0 \quad \text{or} \quad |\nabla \tilde P|\ge 1 + \delta \eps.$$
Now we take $\mu=\delta^2 \eps$ and clearly $\bar \sigma \le \mu^{n+3}$, if $\eps$ is small enough depending on $\delta.$

Finally, we sketch the proof of Theorem \ref{flat}.

\medskip

{\it Proof of Theorem $\ref{flat}.$} Since $u$ is an almost minimizer, then it satisfies \eqref{3000} with $\sigma= \kappa$. Given $\alpha \in (0,\beta/(n+4)]$, let $\eps_0$ depending on $K, \alpha$ be given by Lemma \ref{limp5} and take $\kappa \leq \eps_0^{n+4}.$ Then $u$ satisfies \eqref{eta} and rescaling
we obtain that $u_\eta(x)= u(\eta x)/\eta$ still satisfies \eqref{3000} with $\sigma=\kappa \eta^\beta$ and it is $\eps_0 \eta^\alpha$ flat. Thus we can apply Lemma \ref{limp5} again as long as $\kappa \eta^\beta \leq \eps_0^{n+4} \eta^{\alpha(n+4)}$, which holds in view of our choice of $\alpha$. We conclude that Lemma \ref{limp5} can be applied indefinitely and the theorem follows.\qed

\subsection{Regularity of the free boundary} From Proposition \ref{blowup}, we know that a Lipschitz almost minimizer with small constant $\kappa$ is well approximated by minimizers and this approximation holds also for the free boundaries. On the other hand, the free boundary of a minimizer consists of a singular part which is a closed set of Hausdorff dimension $n-5$, and a regular part which has finite $n-1$ dimension and is locally smooth \cite{AC, CJK, JS}. Thus, using Theorem \ref{flat} and a standard covering argument, we obtain the following result.

\begin{thm}
Let $u$ be a Lipschitz almost minimizer to $J$ in $B_1$ with exponent $\beta$ and sufficiently small constant $\kappa (K,\beta)$, where $K$ is a constant that bounds the Lipschitz norm of $u$. Then
$$\mathcal H^{n-1}(\p \{u>0\} \cap B_{1/2}) \le C, \quad \quad \mbox{with $C$ universal,}$$
and $\p \{u>0\}$ is smooth outside a closed singular set of Hausdorff dimension $n-5$. 
\end{thm}

In the general case we obtain that $$\mathcal H^{n-1}\{ \p \{u>0\} \cap B_{1/2}) \le C(\kappa, \beta,\|u\|_{L^\infty(B_1)}),$$  and $\p \{u>0\}$ is smooth outside a closed singular set of Hausdorff dimension $n-5$. 

We also state the result of \cite{C2,D} about the regularity of Lipschitz free boundaries for the case of almost minimizers. In view of the results in this section, the proof follows with the strategy of \cite{D}.

\begin{thm}\label{thm2} Let $u$ be an almost minimizer in $B_1$ with exponent $\beta$ and constant $\kappa$. Assume that $0 \in F(u)$ and that $F(u)$ is a Lipschitz graph with Lipschitz constant $L.$ Then $F(u) \cap B_{1/2}$ is a $C^{1,\alpha}$ graph, and its $C^{1,\alpha}$ norm is bounded by a constant that depends only on $n,L$ and $\beta$ and $\kappa$.\end{thm}

\end{document}